\documentclass{amsart}
\usepackage[utf8]{inputenc}

\usepackage{multirow, longtable, amsmath, amssymb, amsthm}
\usepackage[mathscr]{euscript}
\usepackage[margin = 1 in]{geometry}
\usepackage{cite}
\usepackage{graphicx}
\usepackage{caption}
\usepackage{float}
\usepackage{listings}
\usepackage{url}

\newtheorem{thm}{Theorem}
\newtheorem{lem}[thm]{Lemma}
\newtheorem{prop}[thm]{Proposition}
\newtheorem{defn}{Definition}

\newcommand{\N}{{\mathbb{N}}}

\title{Syracuse Maps as Non-singular Power-Bounded Transformations and Their Inverse Maps}
\author{I. Assani, E. Ebbighausen, A. Hande}
\address{Idris Assani, University of North Carolina at Chapel Hill}
\email{assani@email.unc.edu}
\address{Ethan Ebbighausen, University of North Carolina at Chapel Hill}
\email{ejebbigh@email.unc.edu}
\address{Anand Hande, University of North Carolina at Chapel Hill}
\email{a8675309@live.unc.edu}
 \subjclass[2020]{Primary 11B75, Secondary 37A40, 37A44}
\begin{document}

\maketitle
\begin{abstract}

We prove that the dynamical system $(\mathbb{N}, 2^{\mathbb{N}}, T, \mu)$, where $\mu$ is a finite measure equivalent to the counting measure, is power-bounded in $L^1(\mu)$ if and only if there exists one cycle of the map $T$ and for any $x \in \mathbb{N}$, there exists $k \in \mathbb{N}$ such that $T^k(x)$ is in some cycle of the map $T$.  This result has immediate implications for the Collatz Conjecture, and we use it to motivate the study of number theoretic properties of the inverse image $T^{-1}(x)$ for $x \in \mathbb{N}$, where $T$ denotes the Collatz map here.  We study similar properties for the related Syracuse maps, comparing them to the Collatz map.  We also analyze some structural properties of the inverse image in relation to asymptotic density of the set $\{x \in \mathbb{N} \mid \exists k \in \mathbb{N}: T^k(x) < x\}$.    

\end{abstract}

\section{\centering{Introduction}}

Many mathematicians have investigated the Collatz Map through a variety of means from number theory to dynamical systems. For a compendium of prior work on the subject, see \cite{Lagabib}. However, this paper seeks to investigate a new route of analyzing this notorious map by generalizing the method introduced by the first author in \cite{Assanipre} to a larger collection of maps, as well as to elucidate the structure of the pre-image trees of some maps with respect to this process.  Mathematicians have also studied several density results related to the Collatz Map.  Both Terras \cite{Terras1979} and Everett \cite{everett2010} were able to show independently that for almost every $x \in \mathbb{N}$, there exists $k \in \mathbb{N}$ such that $T^k(x) < x$, with Korec proving a stronger asymptotic density result, with $T^n(x) < x^c$ for $c > \log_4(3)$ \cite{Korec1994}.  Tao obtained results pertaining to logarithmic density \cite{TTao}. 
Since this topic deals with the positive integers, we shall consider $\mathbb{N}$ not to contain $0$. First, define the truncated Collatz map:

\begin{defn}
The Collatz map is a function  $T: \mathbb{N} \rightarrow \mathbb{N}$ so \begin{equation}T(x) = \begin{cases} \frac{3x+1}{2} & x\,\, odd \\ \frac{x}{2} & x\,\, even \end{cases} \end{equation}
\end{defn}

The Collatz Conjecture states that $\{1,2\}$ is the only cycle of this map, and that the trajectories of the map are bounded. We focus on the latter of these two requirements, which, put rigorously, says that for any $n \in \mathbb{N}$ there exists $m \in \mathbb{N}$ such that $T^{m}(n)$ is part of a cycle. Next, consider a broader class of maps as presented in \cite{Maw1978}:

\begin{defn}
Let a Syracuse-type map $V: \mathbb{N} \rightarrow \mathbb{N}$ be a function of the form \begin{equation}
    V(x) = \frac{m_ix+r_i}{d} \,\,\,\,\,\, if \,\,\,\, x \equiv i \, \textbf{mod \,d}
\end{equation} where $r_i \equiv - im_i \, \, \textbf{mod} \, \, d$ and $i = 0,1,2 ... d-1$, and $gcd(m_0m_1...m_{d-1},d) = 1$
\end{defn}

While the Collatz Map is often referred to as the Syracuse map, we distinguish these general cases as Syracuse-type maps or simply Syracuse maps for brevity. A simple sub-class of these maps often called $px+1$ maps, which are the same as the Collatz map except with some value $p$ replacing the $3$, has often been the focus of study. It is known that several of these maps have cycles, and that some have multiple cycles (\cite{Lagabib}).

\section{\centering{Ergodic Characterization of the General Syracuse Map}}

\subsection{Extending the Hopf Decomposition:}

\bigskip

While it is not completely necessary to use the Hopf Decomposition for the primary result of the next section, it is helpful in providing motivation. Colloquially, the Hopf decomposition provides that, in a non-singular system, the ambient space may be separated into sets conservative and dissipative with respect to the map (for a full proof and rigorous statement, see \cite{Krengel}). In the case of the Collatz system, any unbounded trajectory must be in the dissipative part, while any cycle must be in the conservative part. In fact, this may be sharpened further, and an extended version of this decomposition is given in \cite{Assanipre} for the Collatz map. It may also be expanded to more general maps on the natural numbers.
\bigskip
\begin{thm}
Consider a non-singular dynamical system $(\mathbb{N}, 2^{\mathbb{N}}, V, \nu)$. There exists a partition of $\mathbb{N}$ into sets $C, D_1, D_2$ such that:
\begin{enumerate}
\item{The restriction of $V$ to $C$ is conservative. The set $C$ is $V$-absorbing, and is the at-most-countable union of cycles $C_i$.}
\item{The set $D_1$ is equal to $\bigcup_{k=1}^{\infty}V^{-k}(C) \backslash C$.}
\item{The set $D_2$ is the complement of $C \cup D_1$ in $\mathbb{N}$, $V^{-1}(D_2) = D_2$.}
\item{Any and all unbounded trajectories of $V$ lie in $D_2$}
\end{enumerate}
\end{thm}
\bigskip

\begin{proof}: The Hopf Decomposition provides a partition of $\mathbb{N}$ into sets $C$ and $D$, so that the restriction of $V$ to $C$ is conservative and $C$ is $V$-absorbing. Since $\mathbb{N}$ is countable, $C$ also is, so it can contain at most countable many cycles. This proves (1). Let $D_1 = \bigcup_{k=1}^{\infty} V^{-k}(C)\backslash C$. Let $D_2 = \mathbb{N} \backslash [C\cup D_1]$. Then, $V^{-1}(C\cup D_1) = V^{-1}(C) \cup \bigcup_{k=1}^{\infty} V^{-k-1}(C)\backslash C = C \cup D_1$. Therefore, $V^{-1}(D_2) = D_2$, and (2) and (3) are proven. Finally, given any $x \in \mathbb{N}$, a bounded trajectory means that there exists $n$ so $V^{n}(x) \in C$, so either $x \in C$ or $x \in D_1$. Given an unbounded trajectory of the point $y$, the set $\{y\}$ is wandering, and $V^{m}(y) \notin C$ for any natural number $m$. Thus, $y \in D_2$ and (4) is proven.
\end{proof}

\bigskip

This characterization sets aside the unbounded trajectories of a map in a set $D_2$. Showing that this set is empty would prove half of the Collatz conjecture. This leads to the following characterization that is equivalent to such a case. Furthermore, showing that $C$ is exactly $\{1,2\}$ would prove the entire conjecture.

\subsection{\centering{Dynamical System Characterization:}}

Proving that $D_2$ is empty is essentially the same problem as proving the trajectories bounded, thus it is important to introduce some equivalent criteria for such a case. To do so, consider the following definition.

\begin{defn}: Let $(X, \mathscr{A}, T, \mu)$ be a non-singular dynamical system. It is said to be power bounded in $L^{1}(\mu)$ (or often simply just power bounded) if there exists some $M \in \mathbb{R}_{+}$ such that for all sets $A \in \mathscr{A}$ and natural numbers $n$, $\mu(T^{-n}(A)) \leq M\mu(A)$. \end{defn}
\bigskip

Using this structure from the study of dynamics, the first author presented the following characterization of the Collatz map:
\bigskip

\begin{thm}
Let $(\mathbb{N}, 2^{\mathbb{N}}, T, \mu)$ be the Collatz dynamical system with the counting measure $\mu$. The following are equivalent:
\begin{enumerate}
\item{There exists a finite measure $\alpha$ equivalent to $\mu$ for which the dynamical system $(\mathbb{N}, 2^{\mathbb{N}}, T, \alpha)$ is power bounded in $L^{1}(\alpha)$.}
\item{The set $D_2$ is empty.}
\item{The trajectories of each point $n \in \mathbb{N}$ are bounded.}
\end{enumerate}
\end{thm} 
\bigskip

This characterization may be extended to nearly all maps. Of course, the Collatz map and other similar maps have cycles that define the conjecture surrounding them. It is necessary for these maps to have at least one cycle to use such a characterization, because otherwise the set $D_2$ cannot be empty in any case. Given this caveat, the next theorem follows.
\bigskip

\begin{thm}Let $(\mathbb{N}, 2^{\mathbb{N}}, V, \mu)$ be a dynamical system where $\mu$ is a finite measure equivalent to the counting measure. Then, it is power-bounded in $L^1(\mu)$ if and only if there exists at least one cycle of the map $V$ and for any $x \in \mathbb{N}$, there exists a non-negative integer $k$ such that $V^{k}(x)$ is in some cycle of the map $V$. 
\end{thm}

\begin{proof}
$\Rightarrow$) Consider the power-bounded system $(\mathbb{N}, 2^{\mathbb{N}}, V, \mu)$. By contradiction, let there exist some $x$ so $V^{k}(x)$ is never in a cycle. Then, for all $k \neq l$, $k,l > 0$, $V^{k}(x) \neq V^{l}(x)$. Let $\mu(x) = \delta > 0$.  Further, since the measure is finite, $\mu(\bigcup_{i=0}^{\infty} V^{i}(x)) = \sum_{i=1}^{\infty} \mu(V^{i}(x)) = \epsilon > 0$, implying that $\mu(V^{k}(x)) \rightarrow 0$ as $k \rightarrow \infty$. Given any $M \in \mathbb{R}_{+}$, we may take some large $N$ such that $\mu(V^{N}(x)) < \delta/M$. Hence, $V^{-N}(V^{N}(x)) > \delta > M\mu(V^{N}(x))$, contradicting that the system is power bounded since $M$ was arbitrary.

\medskip

$\Leftarrow$) Let this property hold. Since the space is countable, there are at most countable many cycles, and every point is in a pre-image of a cycle. Let the cycles be the sets $C_1, C_2, ...$. First consider $C_1$. The cycle must be of finite length, N, so we construct a measure on $\mathbb{N}$. Define a function $\mu_1$, which will have a measure value at each specified point, and set the measure of any $A \in 2^{\mathbb{N}}$ to be the sum of the measures of the points in $A$. Let $\mu_1(c) = \frac{1}{2N}$ for any $c \in C_1$, and let  so that $\mu_1(C_1) = \frac{1}{2}$. Next, there are at most countably many points in $V^{-1}(C_1) \backslash C_1$. The infinite case implies the finite case, so we consider this case. Enumerate these points as $\{c_{1}, c_{2}, c_{3} ... \}$. Set $\mu_1(c_{i}) = 2^{-i-3}$, so that $\mu_1(V^{-1}(C_1) \backslash C_1) \leq 2^{-2}$.  Next, consider $V^{-1}(c_j)$, which does not contain $c_j$ nor any other point with a defined measure since such a case would generate a cycle. It again may be enumerated as $\{c_{j,1}, c_{j,2}, c_{j,3} ... \}$, since the set is at most countably infinite. Set $\mu_1(c_{j,i}) = 2^{(-j-2)+(-i-2)}$ so that $\mu_i(V^{-1}(c_j)) \leq 2^{-j-3}$ and so $\mu_1(V^{-1}(V^{-1}(C_1)\backslash C_1)) = 2^{-3}$. Repeat inductively over the pre-images generated by this set, and set all points not in $\bigcup_{i=0}^{\infty}V^{-i}(C_1)$ to have measure zero. 
\bigskip

The process above gives that $\mu_1(\bigcup_{i=0}^{\infty}V^{-i}(C_1)) = \mu_1(C_1) + \mu_1(V^{-1}(C_1) \backslash C_1) + \mu_1(V^{-1}(V^{-1}(C_1) \backslash C_1)) + ... \leq \frac{1}{2} + \frac{1}{4} + \frac{1}{8} + ... = 1$. Repeat this inductively again over the $C_i$, and we may generate a new measure by the countable sum of measures. Let $\mu = \sum_{i=1}^{\infty} 2^{-i-1}\mu_i$, so $\mu(\mathbb{N}) \leq 1$ and $\mu$ is a finite measure. Further, every point in $\mathbb{N}$ is in one of the $\bigcup_{i=0}^{\infty}V^{-i}(C_j)$, so every point has nonzero measure under $\mu$. Next, we need to show that the measure allows our point transformation to be power bounded. By construction, $\mu_i(V^{-1}(A)) \leq \mu_i(A)$ for any set $A$ such that $A \cap C_i = \emptyset$. For a set $B$ intersecting the cycle $C_i$, separating $B = (B \cap C_{i}^{C}) \cup (B\cap C_i)$ gives $\mu_i(V^{-n}(B)) \leq 2^{-n}\mu_i(B \cap C_{i}^{C}) + 2 \mu_i(B \cap C_i) \leq 2\mu_i(B)$ for all $n \in \mathbb{N}$. Therefore, $\mu(V^{-n}(A)) = \sum_{i=1}^{\infty} 2^{-i-1}\mu_i(V^{-n}(A)) \leq \sum_{i=1}^{\infty} 2^{-i}\mu_i(A) = 2\mu(A)$. The measure $\mu$ is finite and power-bounded in $\mathscr{L}^{1}$, as well as equivalent to the counting measure, as desired.

\end{proof}

\bigskip

\textbf{Remark 1}: Given the previous result, it is helpful to note which Syracuse maps have cycles and thus can use this characterization.  For example, consider a map \begin{equation} T(x) = \begin{cases} \frac{px+1}{2} & x \,\, odd \\ \frac{x}{2} & x \,\, even
\end{cases}  \end{equation} where $p$ is an odd number. It is easy to see that $T(1) = \frac{p+1}{2}$, and if $p = 2^{k}-1$, there exists a cycle $\{1, 2^{k-1}, 2^{k-2}, ... 2\}$. The Collatz map falls into this category. It is also known that cycles exist in the case of $5n+1$, and $181n+1$ \cite{10.2307/2006353}.

\section{\centering{Chain Structure in The Collatz Map}}

\indent Looking at the Collatz and generalized Collatz maps in terms of a power-bounded system as above leads naturally to examining the inverse map. The action of the Collatz map is rather erratic when looking in the forward direction, but in the backward direction, some patterns begin to arise, specifically when classifying values as nodes modulo 3. However, these patterns don't necessarily extend to general maps, allowing for classification of these maps.

\subsection{\centering{The Collatz Map:}}

While the map may seem unpredictable at first, looking at values modulo 3 presents some regularity. A simple check shows that $T^{-1}(3p) = \{6p\}$, $T^{-1}(3p+1) = \{6p+2\}$, and $T^{-1}(3p+2) = \{2p+1, 6p+4\}$. For shorthand, we designate three sets.

\begin{defn}
A value $n \in \mathbb{N}_i$ is said to be in $\mathbb{N}_i$ if $n \equiv i \,\, (mod \,\, 3)$. 
\end{defn}

The pre-image of a node in $\N_0$ is again in $\N_0$, the pre-image of a node in $\N_1$ is in $\N_2$, and the pre-image of a node in $\N_2$ includes one node either in $\N_1$ or $\N_0$ and another in $\N_1$. The latter behavior is the most interesting, and may be characterized in a more lucid way, using a simple lemma.

\begin{lem}
Any node in $\N_2$ may be written as $3^a2^bh-1$ where $a>0$, $b \geq 0$, and $h \geq 1$ is not a multiple of $2$ or $3$. 
\end{lem}

\begin{proof}
For arbitrary $n \in \N_2$, it may be written as $3p+2$ be definition, and thus as $3(p+1)-1$. Using the Fundamental Theorem of Arithmetic, $p+1$ has the desired representation $3^{a-1}2^{b}h$. 
\end{proof}

\bigskip

\indent Under this representation, $T^{-1}(3^{a}2^{b}h-1) = \{3^{a-1}2^{b+1}h-1, 2(3^{a}2^{b}h-1)\}$. In particular, we may characterize a family of nodes by either the image or part of the pre-image, where $T(2^{a}h-1) = 3(2^{a-1})h-1$, $T(3(2^{a-1})h-1) = 3^{2}2^{a-2}h-1$, and this repeats until $T^{a}(2^{a}h-1) = 3^{a}h-1$. Define this as a family structure, when such a repeating form occurs. 

\begin{figure}[H]

\centering{\includegraphics[width=4in, height= 2.5in]{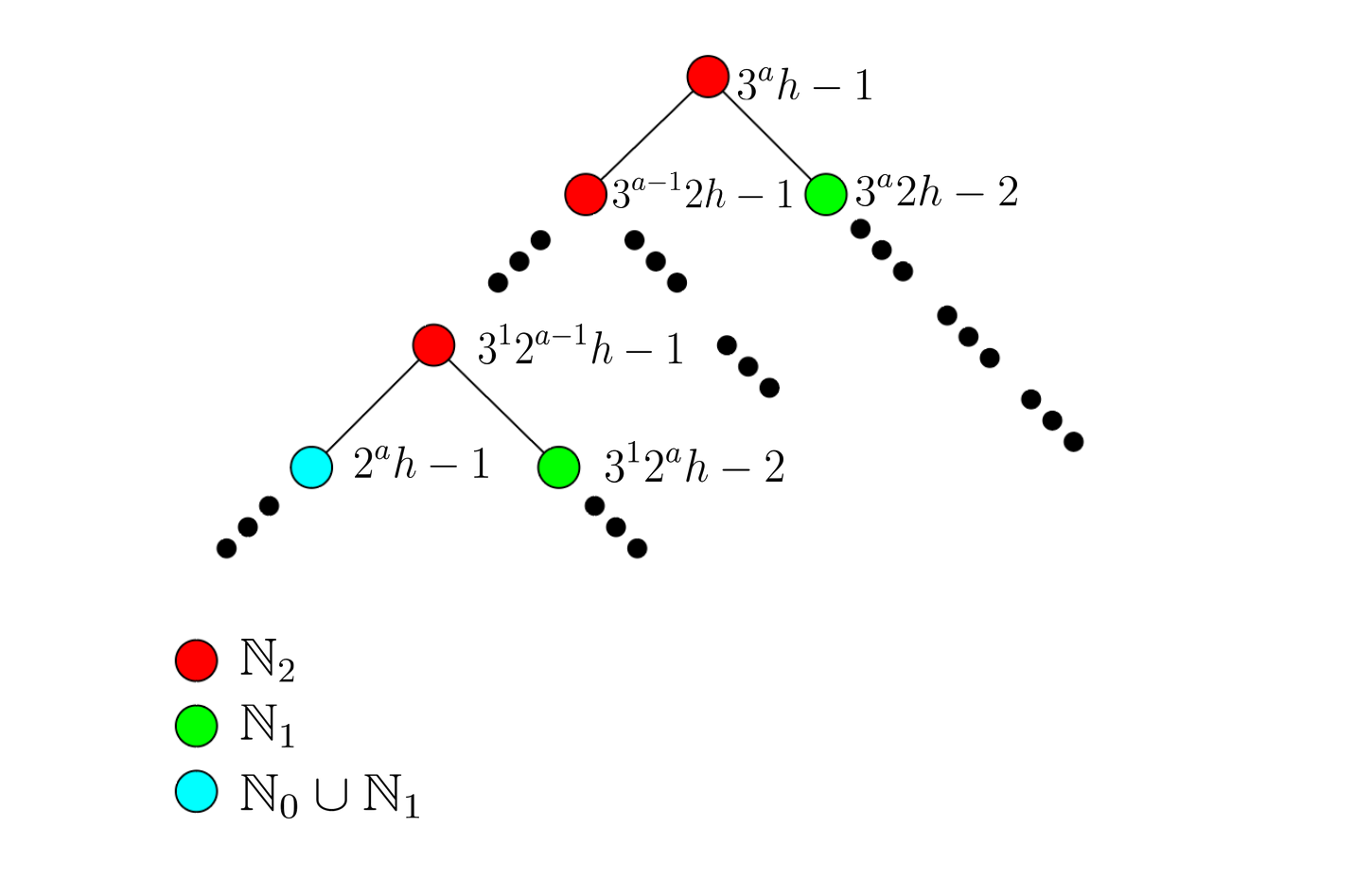}}
\captionsetup{justification=centering}
\caption{Inverse Image Tree generated by $3^ah-1$ up to the $a^{th}$ level}
\end{figure}

In this final case, the node is even and its image is a $\N_1$ node. The pre-image of the $\N_1$ node is again a $\N_2$ node, starting the process again. Connecting these families forms a chain. This chain may extend in the other direction as well. In the case that $2^{a}h-1$ is an $\mathbb{N}_1$ node, its pre-image $2^{a+1}h-2$ is in $\mathbb{N}_2$, and so it takes the form $3^{r}k-1$ following the conditions of the lemma as well. It has been useful to distinguish the nodes of the form $3^{r}k-1$ as ``chain head nodes," although they are on equal footing of importance with the $\mathbb{N}_1$ nodes of the form $2^{a}h-1$. 
\bigskip

\begin{figure}[h]

\centering{\includegraphics[width=4.25in, height= 2.5in]{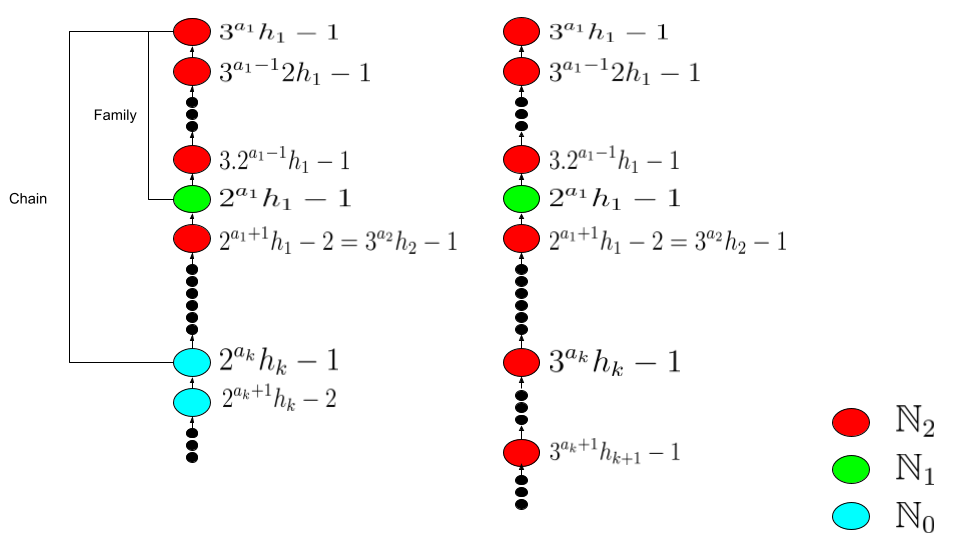}}
\captionsetup{justification=centering}
\caption{A Family and a Chain Illustrated with Nodes}
\end{figure}

\bigskip

\subsection{More General Maps:}

Many of the more general maps with similar forms to the Collatz Map have extremely different structures and properties, for example, that many currently are not known to have a cycle. Consider a sub-class of these maps, where \begin{equation} V(x) = \begin{cases} \frac{px+r}{2} & x \,\, odd \\ \frac{x}{2} & x \,\, even
\end{cases}  \end{equation} where $p$ is an odd integer and $r$ is taken to be an integer so $|r| < p$ and $r$ is relatively prime to $p$. In this case, the idea of a family structure includes some repeating form over a number of nodes. For example, if $T(p^{\alpha}2^{\beta}k-1) = p^{\alpha+1}2^{\beta -1}k-1$. We then have the following.

\begin{thm}
A map $V: \mathbb{N} \rightarrow \mathbb{N}$ has a chain structure precisely when $r = p-2$ or $r = 2-p$.
\end{thm}
\begin{proof}
This relies on a couple building blocks of concepts that led to the chain structure in the Collatz map. First, the family structure gives a node of the form $p^{\alpha}2^{\beta}k-l$ (where $k$ is neither a multiple of $2$ or $p$) so that $V(p^{\alpha}2^{\beta}k-l) = p^{\alpha+1}2^{\beta-1}k-l$. This requires that $l$ be odd, and so \begin{equation} \frac{p(p^{\alpha}2^{\beta}k-l)+r}{2} = p^{\alpha+1}2^{\beta-1}k-l\end{equation} Hence, $\frac{r-pl}{2} = -l$ and $r = (p-2)l$. Since $l$ is an integer, $\frac{r}{p-2}$ is, and since $|r| < p$, $r = \pm (p-2)$. Note that these calculations also show that elements in the class of $\frac{r}{2-p} \,\,(mod \,\, p)$ have two pre-images. These are, in fact, the only nodes with $2$ pre-images. 

This case where $r = \pm(p-2)$ gives a further sub-class of maps with the family structure. Next, we most consider the connection of families to create chains. Using the previous notation, let $m$ be the largest positive integer so $2^m$ divides $p^{\alpha+\beta}k + \frac{r}{2-p}$. We would require that \begin{equation} V^{m+1}(p^{\alpha+\beta}k+\frac{r}{2-p}) \equiv \frac{r}{2-p} \,(mod\,\,p)  \end{equation} Such a condition would connect two families, since $V^{m+1}(p^{\alpha+\beta}k+\frac{r}{2-p})$ would then have a representation $p^{\alpha_2}2^{\beta_2}k_2-\frac{r}{p-2}$.

The above equation says that \begin{equation}\frac{p(\frac{p^{\alpha+\beta}k-\frac{r}{p-2}}{2^m})+r}{2} \equiv \frac{r}{2-p} \,(mod \,\,p) \end{equation} Since $p$ is odd, this occurs exactly when \begin{equation}(2-p)(p^{\alpha+\beta+1}k-\frac{pr}{p-2}+2^{m}r) \equiv 2^{m+1}r \,(mod\,\,p) \end{equation} which reduces to the trivial \begin{equation}2^{m+1}r \equiv 2^{m+1}r \,(mod\,\,p) \end{equation} We thus have that the families connect for this form of map whenever they occur. \end{proof}
\bigskip

\textbf{Remark 2:} We may extend further to the case of \begin{equation}
    V(x) = \frac{m_ix+r_i}{d} \,\,\,\,\,\, if \,\,\,\, x \equiv i \, \textbf{mod \,d}
\end{equation} where $r_i \equiv im_i \, \, \textbf{mod} \, \, d$ and $i = 0,1,2 ... d-1$, and $gcd(m_0m_1...m_{d-1},d) = 1$, as proposed in \cite{Maw1978} and \cite{MaB1990}. In this case, since $r_i = im_i$, when $\frac{r_i}{m_i-d}$ is an integer, $T(m_{i}^{\alpha}d^{\beta}k-\frac{r_i}{m_i-d}) = m_{i}^{\alpha+1}d^{\beta-1}k-\frac{r_i}{m_i-d}$. These families may occur in one such class modulo $d$, or in several depending on this condition. Future results using the chain structure on the Collatz map may be generalized to such a class of maps.
\bigskip

\section{The Set $D_2$ and Previous Density Results}

The properties of the set $L = \{x \in \mathbb{N} \mid \exists k \in \mathbb{N}: T^k(x) < x\} \subset \mathbb{N}$ are of interest, as the Collatz Conjecture is equivalent to the statement $L = \mathbb{N} - \{1\}$. The set $L$ was introduced by R. Terras in \cite{Terras1976}, where it was also shown to have density 1. The works of Korec generalized this set to $M_c = \{x \in \mathbb{N} \mid \exists k \in \mathbb{N}: T^k(x) < x^c\}$ for $c > log_3 (4)$ in \cite{Korec1994}, which also has density 1, and further generalizations follow in Lagarias \cite{LagDen}, Everett \cite{everett2010}, and Tao \cite{TTao}.

\medskip

However, these density results do not have immediate implications to the structure of $D_2$, the elements of $\mathbb{N}$ with unbounded trajectories, or $\mathbb{N} - D_2$, since $M_c$ and $L$ could contain elements which have unbounded trajectories as long as they have some lesser iterated image. 

\begin{prop}

Assume that $D_2 \neq \emptyset$. Then, 
\begin{itemize}
\item{1) There exists $y \in D_2 \cap M_c$ for any $0 < c \leq 1$.}
\item{2) Assume $D_2$ has density $0$, and that there are finitely many cycles $C_1,...,C_n$.  Let $f: \mathbb{N} \to \mathbb{R}$ satisfy $f(x) > \max_{1 \leq i \leq n}(\min(C_i))$ for all $x \in \mathbb{N}$.  Then $L_f := \{x \in \mathbb{N} \mid \exists k \in \mathbb{N}: T^k(x) < f(x)\}$ has density $1$. }
\item{2) $D_2$ can be written as a countable disjoint union of sets of the form $B_x := \bigcup_{n \in \mathbb{N}} \bigcup_{k \geq 0} T^{-n}(T^k(x))$ where each $x$ is in $D_2$.  Furthermore, each $B_x$ is invariant under $T$.}

\end{itemize}

\end{prop}
\begin{proof}

(1) Let $D_2$ be nonempty. By the well-ordering of the natural numbers, there exists some minimum element $d$. Consider that $2^{a}d \in D_2$ for any $a \in \mathbb{N}$ and in particular $T^{a}(2^{a}d) = d$. Fix some $c \in (0,1]$. Then, Take $a$ large enough such that $(2^{a}d)^{c} > d$, and then $T^{a}(2^{a}d) < (2^{a}d)^{c}$.
\bigskip

(2) Since $D_2$ has density $0$, $\mathbb{N} - D_2$ has density $1$.  So it suffices to show that $\mathbb{N} - D_2 \subset L_f$.  If $x$ has bounded trajectories, then by definition for some $j \in \mathbb{N}$ we have $T^j(x) \in C_i$ for some $i = 1,...,n$.  By taking more iterates, we have $T^k(x) = \min(C_i) < f(x)$.

\bigskip
(3) We clearly have $D_2 = \bigcup_{x \in D_2} B_x$ which is a countable union since $D_2 \subset \mathbb{N}$ is countable.  Since $D_2$ is invariant under $T$ and $y \in B_x$ implies $B_y \subset B_x$, we can discard some elements in the above union to yield a disjoint union.  By definition, $B_x$ is the smallest set that contains all preimages of all images of $x$ under $T$ so $B_x$ is invariant.  

\end{proof}


\medskip

\textbf{Remark 3:} With these assumptions, we would obtain stronger density results for the likes of functions $f$ previously analyzed.  In addition, while this set $D_2$ is not as well-understood as $L$, Theorem 1 shows that it has invariance under the forward and reverse Collatz map, $T(D_2) \subset D_2$ and $T^{-1}(D_2) \subset D_2$.  Similarly, if any number $x$ has bounded trajectories, then so does $T(x)$ and so do the values in the set $T^{-1}(x)$.  Such invariance properties are not immediate for $L$:  If $T^k(x) < x$ for some $x,k \in \mathbb{N}$ and $x$ is even, then it is not immediate that $T^k(x) < \frac{x}{2}$, or if any iterate is less than $\frac{x}{2}$.  If $x = 3p+2 \in \mathbb{N}_2$, then $T^k(x) < x$ does not imply $T^k(x) < 2p+1 \in T^{-1}(x)$.  This motivates the study of the structure of $D_2$, assuming it is non-empty.  

\bigskip
\bigskip

Part 3 of the above proposition gives a way to organize $D_2$ in relatively nice, invariant chunks $B_{x_i}$ which generate $D_2$. Following the family and chain structures introduced previously, the values in $\mathbb{N}_2$ and how they behave under $T^{-1}$ determine the structure of the sets $B_{x_i}$ and hence $D_2$. Given that the preimages of $3^{k}h-1$ expand under the action $\frac{2x-1}{3}$ exactly $k$ times before reaching a non-$\mathbb{N}_2$ node, we then focus on this preimage level.

\begin{defn}
We refer to the nodes in $\bigcup_{i=0}^{k} T^{-i}(3^{k}h-1)$ as the triangle generated by the node $3^{k}h-1$ and the nodes $n \in T^{-a}(3^{k}h-1)$ as the $a$-th level preimages.
\end{defn}

\medskip

For example, the triangle generated by $9h - 1$ is just the first $3$ levels of the tree generated by $9h-1$, i.e $\{9h-1\},\{6h-1, 18h-2\}, \{4h-1, 12h-2, 36h - 4\}$. By a quick computation, one can see that the only possible elements of the triangle which are not in $L$ are iterates of the $\frac{2x-1}{3}$ action (specifically, the nodes $9h-1, 6h-1, 4h-1$). For brevity, let us consider the two possible actions of the inverse map $T^{-1}$. Call the $\frac{2x-1}{3}$ the left action, and the action $2x$ the right action. Thus, the iterates of left actions are on the leftmost branch of the triangle (see Figure 1). This leads to the following claim:

\medskip

\textbf{Claim}: For all $x$ in a triangle generated by $3^k h - 1$, we can find $m \in \mathbb{N}$ such that $T^m(x) < x$, if $x$ does not lie on the left branch of the triangle.  

\medskip

This statement has been verified for several different values of $a,h \in \mathbb{N}$ and $h \notin \mathbb{N}_0$.  In particular, we have verified this statement for all possible combinations of $a \in \{1,2,...,20\}$ and $h \in \{1,2,...,300\} - \mathbb{N}_0$, through the code in the appendix.

 \medskip

The numerical data suggests that this claim is true, which would give a more precise characterization of the elements of $D_2$.    

\section{Code Appendix}

\begin{lstlisting}


def t(n):
  if n % 2 == 0:
    return n/2
  else:
    return (3*n + 1)/2

def t_iter(n,k):
  ls = [n]
  next = n
  for i in range(1,k+1):
    var1 = t(next)
    next = var1
    ls.append(var1)
  return ls
def inv_t(n):
  if n % 3 != 2:
    return [2*n]
  else:
    p = (n-2)/3
    return [2*p + 1, 2*n]

def flatten(l):
    return [item for sublist in l for item in sublist]

def generateTree(a,h):

  num = 4*(3**a * h - 1)
  flag = True

  tree = [[num]]
  current_lvl = [num]
  for i in range(1, a+1):
    next = flatten([inv_t(x) for x in current_lvl])
    tree.append(next)
    current_lvl = next
  return tree

def leftBranch(a,h):
  num = 3**a * h - 1
  return [3**(a-k) * (2**k * h) - 1 for k in range(0,a+1)]

###Run this method to check statement for different a,h values.  
def testTree(a,h): 
  tree = generateTree(a,h)
  branch = leftBranch(a,h)

  for k in range(2, len(tree)):
    lvl = tree[k]
    for node in lvl:
      t_image = t_iter(node, k)[1:]
      if (min(t_image) > node and node != lvl[0]):  
      #Checks if the node doesn't shrink and if it's not on the left hand branch
        return (False, node, k)
      

  return (True, None, None)
\end{lstlisting}

\bibliographystyle{plain}
\bibliography{Ref}{}
\bigskip
\bigskip
\bigskip

\end{document}